\documentclass[a4paper,11pt]{amsart}
\usepackage{amssymb, amsmath}
\usepackage{amscd}
\numberwithin{equation}{section}
\usepackage{epsfig}
\usepackage{amsmath}
\usepackage{amsfonts,amssymb,amsopn}

\usepackage{amsthm}
\usepackage{hyperref}
\usepackage{verbatim}
\usepackage{color}
\usepackage[color,all]{xy}
\usepackage{graphicx}

\newcommand{\Ker}{\mathrm{Ker}}

\newcommand{\im}{\mathrm{Im}\,}

\newcommand{\Ext}{\mathrm{Ext}}
\newcommand{\rk}{\mathrm{rk}}
\newcommand{\Coker}{\mathrm{Coker}}

\newtheorem{theorem}{{\textbf Theorem}}[section]

\newtheorem{corollary}[theorem]{{\textbf Corollary}}
\newtheorem{lemma}[theorem]{{\textbf Lemma}}

\newtheorem{defit}[theorem]{{\textbf Definition}}
\newtheorem{remit}[theorem]{{\textbf Remark}}

\newenvironment{remark}{\begin{remit}\rm}{\end{remit}}

\title{Excess dimensions of Brill--Noether schemes of rank two stable bundles}

\author{Ali Bajravani}
\address[]
{Department of Mathematics, Faculty of Basic Sciences, Azarbaijan Shahid Madani University, Tabriz, I.\ R.\ Iran., P.\ O.\ Box: 53751-71379.}
\email{bajravani@azaruniv.ac.ir}

\address{School of Mathematics, Institute for Research in Fundamental Sciences (IPM), P.O. Box: 19395-5746. \ Tehran, Iran.}


\begin{document}

\begin{abstract} 
	We use results of M. Aprodu and E. Sernesi to extend a result by Fulton--Harris--Lazarsfeld in Brill--Noether theory of line bundles 
	to Brill--Noether loci of stable bundles inside the moduli space of rank two stable vector bundles on a smooth projective algebraic curve. As a consequence; if $B^k_{2, d}$ is of expected dimension, then we prove that its smoothness is equivalent with the smoothness of $B^{k}_{2, d+1}$.
 \end{abstract}

\maketitle

\vspace{-0.4cm}

\section{Introduction}
Let $C$ be a smooth projective algebraic curve of genus $g$. For integers $r$ and $d$ with $0\leq 2r\leq d$,
the 
Brill--Noether scheme of line bundles, $W^r_d$, parameterizes line bundles of degree $d$ with the space of global sections of dimension at least $r+1$. 
Brill--Noether theory (see for example \cite{GH}) computes the expected dimensions of these schemes in terms of $d$, $r$ and $g$, and shows that this dimension
	is attained for a general curve. However, the dimension may be strictly
	larger when $C$ is special in the scense of moduli. 
The basic Martens' theorem bounds the dimension of $W^r_d$ in terms of
$d$ and $r$, on arbitrary smooth projective curves. Variations of this basic result have been obtained in the literature, see \cite{Ox}, \cite{Capo10}, \cite{Baj15} and \cite{Baj19}. See also \cite[Theorem 3.5]{BH} and \cite[Theorem 3.7]{BH}.
Another result concerning dimensions of Brill--Noether schemes of line bundles, which holds on arbitrary curves,
is the excess dimension inequality.
William Fulton, Joe Harris and Rob Lazarsfeld proved in \cite{FHL} that the difference of dimensions of $W^r_
d$ and $W_d^{r+1}$, (resp., of $W_d^r$ and $W_d^{r+1}$), are
controllable in terms of $r$ (resp., in terms of $r$, $d$ and $g$).
A significant consequence of their result is that if $\dim W^r_d\geq r+1$, then $W^r_{d-1}$ would be non-empty.

 Marian Aprodu and Edoardo Sernesi have extended the excess dimension inequalities to the schemes of secant loci of very ample line bundles. See \cite{AS}. 
Michael Kemeny used their result in studying the syzygies of curves of arbitrary gonalities, see \cite{Kemeny}. As well,  in \cite{Baj17}, the author uses Aprodu-Sernesi' excess dimension result to obtain a Mumford type theorem for schemes of secant loci.

Since some bundles 
appeared in constructing $W^r_d$'s turns out to be ample, no non-emptiness prerequisite is needed in Fulton-Harris-Lazarsfeld framework. But, this positivity property fails to be hold in Aprodu and Sernesi' general setting. So they enter a non-emptiness assumption, which is slightly a weaker hypothesis. 

If $X$ is an integral algebraic scheme and
	{\small\begin{align}\label{ASdiagram}
	\xymatrix {
		0\ar[r] & \mathbb{L} \ar[r] &\mathbb{F}_1\ar[r]^{\pi}&\mathbb{F}\ar[r] & 0,\\
		&&\mathbb{E}\ar[u]^{\sigma}\ar[ur]_{\pi\circ\sigma},&&
	}
	\end{align}
is an arbitrary diagram of bundles on $X$ with $\rk (\mathbb{E})=m, \rk(\mathbb{F})=n, \rk(\mathbb{F}_1)=n+1$, then
 Aprodu and Sernesi compare dimensions of $D_k(\sigma)$, 
$D_{k+1}(\sigma)$ and $D_k(\pi\sigma)$, where for a morphism $\sigma: E\rightarrow F$ of vector bundles on $X$, the loci 
$D_k(\sigma)$ is defined as 
\begin{displaymath}
D_k(\sigma):=\{x\in X\mid \rk(\sigma_x)\leq k\}.\end{displaymath}
Consequently Fulton--Harris--Lazarsfeld's result is a special case of Aprodu--Sernesi' result. 

We use Aprodu and Sernesi' general machinery to study the excess-dimension problem for Brill--Noether loci of stable vector bundles of rank $2$ and degree $d$ with the space of global sections at least of dimension $k$ inside the moduli space of rank two stable vector bundles, $U(2, d)$.   
The closure of some suitable loci, denoted by $B^1_{2, d+1, p}$, consisting of general dual elementary transformations of general bundles in $B^1_{2, d+1}$ is  
 a divisor in 
 $U(2, d+1)$ for $1\leq d\leq 2(g-1)-1$. On a suitable dense sub-locus $\mathcal{U}\subseteq U(2, d+1)$, consisting of general dual elementary transformations of general bundles in $U(2, d)$, we obtain a diagram of vector bundles. As Lemma \ref{Stability lemma}(i) indicates, there exists a dense subset of $U(r, d+1)$ consisting of general dual elementary transformations of general bundles in $U(r, d)$. But it is not known if this loci intersects the loci $B^k_{r, d}$ for $r\geq 3$. As for the case $r=2$, we observe that such a dense subset has non-empty intersection with any non-empty component of $B^k_{2, d}$. See lemma \ref{Stability lemma}(ii).
So comparing dimensions of degeneracy loci of the diagram gives a right comparision of dimensions of various $B^k_{2, d}$'s.
 The diagram we use, diagram \ref{maindiagram}, is a \textquotedblleft dual\textquotedblright
   version of Aprodu-Sernesi' diagram (\ref{ASdiagram}). 
   We apply theorem \ref{SDual}, by which we obtain our main result. See Theorem (\ref{Main B-N1}).  
Using theorem \ref{Main B-N1} together with a comparision of dimension of tangent spaces we relate the smoothness 
 of $B^{k}_{2, d}$ to those of $B^{k}_{2, d+1}$. See corollary \ref{cor}.
\subsection*{Notation} 
For a sheaf $F$ on $C$, we abbreviate $H^i ( C, F )$ and $h^i ( C, F )$ to $H^i ( F )$ and $h^i ( F )$, respectively. If $\sigma: E\rightarrow F$ is a morphism of locally free sheaves $E$, $F$ of ranks $\alpha, \beta$ on an algebraic variety $X$,  respectively, then for $p\in X$ the $\alpha\times \beta$ matrix $\sigma\mid_p$ acts on an $\alpha$-vector $\nu\in E\mid_p$ by $^t \nu\cdot\sigma\mid_p$.
\section{Excess dimensions of Brill-Noether loci of rank two stable bundles}\label{sec3}
\subsection{A dual version of Aprodu and Sernesi' result}\label{AS-Section}
Let $X$ be an integral algebraic variety and consider a diagram of vector bundles
	{\small\begin{align}\label{maindiagram}
	\xymatrix {
		0\ar[r] & \mathbb{F} \ar[dr]_{\sigma i}\ar[r]^{i} &\mathbb{F}_1\ar[r]^{\pi}\ar[d]^{\sigma}&\mathbb{L}\ar[r] & 0,\\
		&&\mathbb{E}.&&
	}
	\end{align}
}
\begin{theorem}\label{SDual}
Suppose $\mathbb{E}, \mathbb{F}_1$ and $\mathbb{F}$ are vector bundles of ranks $m$, $n+1$ and $n$, respectively, which are set in a diagram
as \ref{maindiagram}.	Assume moreover that no irreducible component of $D_k(\sigma i)$ is contained in $D_{k}(\sigma)$. \\
\begin{enumerate}
	\item{ If $D_k(\sigma i)$ is nonempty, then
$$\dim D_k(\sigma i) \ \ge  \dim D_{k+1}(\sigma)-(n-k),$$}
\item{ If $D_k(\sigma)$ is nonempty, then
$$	\dim D_k(\sigma)\geq \dim D_k(\sigma i)- (m-k),$$
$$	\dim D_k(\sigma)\geq \dim D_{k+1}(\sigma)-(m+n-2k).$$
}
\end{enumerate}	
\end{theorem}
\begin{proof}
	For a morphism $\gamma: E\rightarrow  F$ of vector bundles on an integral algebraic scheme $X$, one has $D_k(\gamma)=D_k(\gamma^*)$, where $\gamma^*:F^*\rightarrow  E^*$ is the dual homomorphism of $\gamma$. 
	The theorem follows, because diagram \ref{maindiagram} is a dual version of the diagram considered in proposition \cite[Th. 3.1]{AS}. 
\end{proof}
\subsection{A divisorial component in $B^1_{r, d+1}$}\label{3.0.1}
Assume that $g\geq 2$. Let $U(r, d)$ denote the moduli space of stable vector bundles of rank $r$ and degree $d$. Then, $U(r, d)$ admits an $\acute{e}$tale finite cover on which there exists a Poincar$\acute{e}$  bundle. Precisely, there exists an $\acute{e}$tale finite cover $\pi: \mathcal{U}\rightarrow U(r, d)$ and a bundle $\mathcal{E}_d\rightarrow \mathcal{U}$ such that for each $E\in \mathcal{U}$ one has $\mathcal{E}_{d\mid \lbrace E \rbrace\times C}\cong \pi(E)$. Furthermore, for any $\lambda \in \mathbb{N}$,
$U(r, d)\cong U(r, d+\lambda r)$ by an isomorphism $\theta$, where 
$\theta: U(r, d+\lambda r) \rightarrow U(r, d)$ is given by $E\mapsto E\otimes \mathcal{O}(-D)$ with $D:=p_1+\ldots +p_\lambda,$ for some fixed $p_i\in C$. 

For an integer
$\delta$ with $\delta:=d+\lambda r > 2r(g-1)$, consider the projections
\begin{align}\label{diag1}
\xymatrix {
	\!\!\!\!\!\!\!\!\!\!\!\!\!\!\! \ar[d]^{\pi_2} U(r, \delta)\times C \ar[r]^{\pi_1}&C\\
	U(r, \delta),&
}
\end{align}
and set $\mathbb{E}_\delta:=(\pi_2)_*(\mathcal{E}_\delta)$, $\mathbb{F}_\delta:=(\pi_2)_*(\mathcal{E}_\delta\otimes \pi_1^*\mathcal{O}_D(D))$.
The bundles $\mathbb{E}_\delta$ and $ \mathbb{F}_\delta$ would be of ranks $\delta-r(g-1)$ and $\lambda r$, respectively. See \cite{G-T}.

The loci of stable vector bundles $E$, of rank $r$ and degree $d$ with the space of global sections of dimension at least $k$ is up to the isomorphism $\theta$ the $(\delta-r(g-1)-k)$-th degeneracy locus of the evaluating morphism of bundles
\begin{align}\label{gamma}
\gamma_\delta: \mathbb{E}_\delta\rightarrow \mathbb{F}_\delta.
\end{align}
Observe that for $E\in U(r, \delta)$ with $\delta > 2r(g-1)$ one has $h^0(E)=\delta-2r(g-1)$ and so the locus $E\in U(r, \delta)$ with $\delta > 2r(g-1)$ and $h^0(E)=1$ might be empty. By abuse of notations, we denote by $B^1_{r, \delta}$ the loci $\theta^{-1}(B^1_{r, d})$. 	

Recall that in an extension of vector bundles $0\rightarrow E \rightarrow F \rightarrow \mathcal{O}_p \rightarrow 0$, the bundle $E$ is called an elementary transformation of $F$ by $p$ and $F$ is called a dual elementary transformation of $E$ by $p$. 

The Segre invariant $s_1(E)$ is defined to be
\begin{align}
s_1(E) := \min \{\deg(E) - 2 · \deg(L) : L \subset \mbox{E a line subbundle}\}.
\end{align}
Then E is stable if and only if $s_1(E) \geq 1$. 
\begin{lemma}\label{Stability lemma} 
	
(i) Assume that $1\leq d\leq n(g-1)-1$ and $p\in C$ is generic. Then, a general extension as $0 \rightarrow E \rightarrow F \rightarrow \mathcal{O}_p \rightarrow 0,$ with $E$ a general 
element of $B^i_{r, d}$, $0\leq i\leq 1$, is stable. 

(ii) Suppose $E$ is an arbitrary stable rank two vector bundle and $p\in C$ generic. Then, a general dual elementary transformation of $E$ by $p$ is stable.  

(iii) Assume $p\in C$ is generic and $E$ an arbitrary stable rank two vector bundle. Then,
 a general elementary transformation of $E$ by $p$ is stable. 
\end{lemma}
\begin{proof} (i) Observe that for $0\leq d\leq r(g-1)$, the loci $B^1_{r, d}$ and $B^0_{r, d}:=U(r, d)$ are irreducible by \cite[Theorem II.3.1]{S}. So,
according to the openness of stability, it is enough to prove that there exists one special extension in each of them with the stated property.	

We prove the case $i=1$. The other case is similar. Set $d=r\alpha +t$ and observe then that $\alpha \leq g-1$. Take line bundles $L_1, \ldots, L_{r-1}$ in the loci 
$B^0_{1, \alpha}\setminus B^1_{1, \alpha}$ and a line bundle $H\in B^1_{1, \alpha}\setminus B^2_{1, \alpha}$ with the property that no two of these line bundles are isomorphic.
 For generic points $p_1, \ldots, p_{t}$ in $C$, a general extension as
\begin{align}\label{stabilitydiag}
0 \rightarrow \left(\bigoplus _{i=1}^{r-1} L_i\right)
\oplus H \rightarrow F \rightarrow \left(\bigoplus _{i=1}^{t} \mathcal{O}_{p_i}\right)\oplus \mathcal{O}_{p} \rightarrow 0,
\end{align}
is stable by the Theorem of Mercat. See \cite[p. 76]{Me}.  
A general extension $F$ as in \ref{stabilitydiag} is mapped by $\mu$ to a general element of 
$\Ext^1\left((\bigoplus _{i=1}^{t} \mathcal{O}_{p_i}), (\bigoplus _{i=1}^{r-1} L_i)
\oplus H\right)$, where 
$$\mu: \Ext^1\left((\bigoplus _{i=1}^{t} \mathcal{O}_{p_i})\oplus \mathcal{O}_{p}, (\bigoplus _{i=1}^{r-1} L_i)
\oplus H\right) \longrightarrow \Ext^1\left((\bigoplus _{i=1}^{t} \mathcal{O}_{p_i}), (\bigoplus _{i=1}^{r-1} L_i)
\oplus H\right)$$
is the pull back morphism of the inclusion of sheaves
$\bigoplus _{i=1}^{t} \mathcal{O}_{p_i} \rightarrow \left(\bigoplus _{i=1}^{t} \mathcal{O}_{p_i}\right)\oplus \mathcal{O}_{p}$.
Observe that $\mu$ is the second connecting homomorphism in the long exact cohomology sequence associated to the exact sequence 
$$0 \rightarrow \bigoplus _{i=1}^{t} \mathcal{O}_{p_i} \rightarrow \left(\bigoplus _{i=1}^{t} \mathcal{O}_{p_i}\right)\oplus \mathcal{O}_{p}\rightarrow \mathcal{O}_p \rightarrow 0.$$	
So $\Coker(\mu) \subseteq \Ext^2\left((\bigoplus _{i=1}^{t} \mathcal{O}_{p_i}), (\bigoplus _{i=1}^{r-1} L_i)\oplus H\right)$ and it has to vanish. Consequently $\nu$ would be surjective.
Therefore, $\dim \im \mu = \dim \Ext^1\left((\bigoplus _{i=1}^{t} \mathcal{O}_{p_i}), (\bigoplus _{i=1}^{r-1} L_i)
\oplus H\right)$. 

Since, again by Mercat's theorem, a general element of $\Ext^1\left((\bigoplus _{i=1}^{t} \mathcal{O}_{p_i}), (\bigoplus _{i=1}^{r-1} L_i) \oplus H\right)$ is stable, we obtain an extension as $0 \rightarrow E \rightarrow F \rightarrow \mathcal{O}_p \rightarrow 0,$ with $E$ and $F$ stable. 

The genericity of $p_i,$ $p$ with $1\leq i\leq t$ and the genericity of $E$ together with lemma \cite[2.5]{N} implies that $h^0(E)=h^0\left((\bigoplus _{i=1}^{r-1} L_i)
\oplus H\right)=h^0(F)$. So $h^0(E)=h^0(F)=1$, as required.

(ii) Either the Segre invariant $s_1(E)\geq 2$ or $s_1(E)=1$.  
If $s_1(E) \geq 2$, then for
any $L \subset E $ we have
$$(\deg(E)+1)-(2 · \deg(L)-2) \geq 1;$$
that is,
$$\deg(F) - 2 · \deg(L(p)) = s_1(F) \geq 1,$$
for any elementary transformation $0\rightarrow  E \rightarrow F \rightarrow \mathcal{O}_p \rightarrow 0$. Thus $F$ is stable.

If $s_1(E) = 1$, then by \cite[Prop. 4.2]{L-Na} the number of maximal line subbundles of $E$ is finite. Let $0\rightarrow  E \rightarrow F \rightarrow \mathcal{O}_p \rightarrow 0$ be an elementary transformation with $\Ker(E_{\mid_p} \rightarrow F_{\mid_p}) \neq L{\mid_p}$ for any maximal subbundle $L$. So if $M \subset E$ is a line subbundle with $M_{\mid_p} = \Lambda$, then $\deg(E) - 2 · \deg(M) \geq 3$. As above, $(\deg(E)+1)-(2 · \deg(M)-2) \geq 1;$ that is,
$$\deg(F) - 2 · \deg(M(p)) = s_1(F) \geq 1.$$
	Also, if $L$ is a maximal subbundle of $E$, then by the choice of $\Lambda$, also $L$ is a subbundle of
	$F$, and
	$$\deg(F) - (2 · \deg(L)) \geq 2.$$
	Thus $s_1(F) \geq 1$. Hence a general elementary transformation $0\rightarrow  E \rightarrow F \rightarrow \mathcal{O}_p \rightarrow 0$ is a stable
	vector bundle.
	
(iii) This is immediate by dualizing (ii). 
\end{proof}
Motivated by Lemma \ref{Stability lemma}, for a generic point $p\in C$ and $i=0, 1$, we set
{\small
	\begin{align}\label{convention}
B^i_{r, d+1, p}:=\lbrace F\in B^i_{r, d+1}\mid F\!\!\!\!\!\quad \mbox{is a general extension}\!\!\!\!\!\quad
0\rightarrow E \rightarrow F \rightarrow \mathcal{O}_p \rightarrow 0, \quad \!\!\!\!\! E\in B^i_{r, d}  \rbrace,\end{align}
where we have considered $B^1_{r, d}$ the same as its preimage via $\acute{e}$tale map $\mathcal{U}\rightarrow U(r, d)$.

 Furthermore, we set
\begin{align}\label{convention1}
B^0_{r, \delta+1, p}:=\theta^{-1}(B^0_{r, d+1, p}) \quad ,\quad B^1_{r, \delta+1, p}:=\theta^{-1}(B^1_{r, d+1, p}).   \qquad \qquad \qquad
	\end{align}
Observe that the superscript numbers $0$ and $1$ in $B^0_{r, \delta+1, p}$ and $B^1_{r, \delta+1, p}$, respectively, don't refer to the number of linearly independent global sections of the bundles therein. 

Taking Lemma \ref{Stability lemma} into account one obtains $\overline{B^0_{r, \delta+1, p}}=U(r, \delta +1)$. As for $\overline{B^1_{r, \delta+1, p}}$ we have the following	
\begin{theorem}\label{lem 2-2}
Suppose $1\leq d\leq n(g-1)-1$. Then, the loci $\overline{B^1_{r, \delta+1, p}}$ is an irreducible closed subscheme of $B^1_{r, \delta+1}$ of codimension $1$.
\end{theorem}
\begin{proof}
Since $B^1_{r, \delta+1, p}=\theta^{-1}(B^1_{r, d+1, p})$ by our convention in \ref{convention} and \ref{convention1}, 
Lemma \ref{Stability lemma} implies that $B^1_{r, \delta+1, p}$ is a non-empty subset of $B^1_{r, \delta+1}$, so $\overline{B^1_{r, \delta+1, p}}$ would be non-empty as well. 
Restrict the Poincare bundle $\mathcal{E}_\delta$ to $B^1_{r, \delta}\times C$ and denote its restriction by the same letter. Set
	\begin{align}\mathbb{V}_{\delta}:=R^1(\pi_1)_*\left(\mathcal{E}xt^1\left((\pi_{2, \delta})^*\mathcal{O}_p, \mathbb{E}_\delta\right)\right),
	\end{align}
	where $\mathcal{E}xt^i(F, G)$ denotes the $i-th$ \textquotedblleft ext\textquotedblright
	 sheaves of $F$ and $G$. The sheaf $\mathbb{V}_{\delta}$ is a vector bundle on $B^1_{r, \delta}$ of rank $r$, by \cite[pp. 166-167]{ACGH85}; and the projective bundle 
	\begin{align}\label{alpha}
	\alpha_\delta: \mathbb{P}(\mathbb{V}_{\delta})\rightarrow B^1_{r, \delta}\end{align}
	is the family of $\Ext^1(\mathcal{O}_p, E)$'s as $E$ varies in an open subset of $B^1_{r, \delta}$. Then, one has
	$$\overline{B^1_{r, \delta+1, p}}=\im (\nu_{\delta}),$$
	where the morphism $\nu_{\delta}: \mathbb{P}(\mathbb{V}_{\delta})\rightarrow B^1_{r, \delta+1}$ assigns the stable bundle $F$ to each extension $0\rightarrow E \rightarrow F \rightarrow \mathcal{O}_p \rightarrow 0$. 
	Now, in order to prove that $\overline{B^1_{r, \delta+1, p}}$ is of codimension $1$ in $B^1_{r, \delta+1}$, observe that the
	fibers of both morphisms $\alpha_\delta$ and $\nu_\delta$ are $r$-dimensional. Indeed, for an extension $e: 0\rightarrow E \rightarrow F \rightarrow \mathcal{O}_p \rightarrow 0$ and for $E\in B^1_{r, \delta}$, we have
	$$\dim\nu_\delta^{-1}(e)=\dim \Ext^1(\mathcal{O}_p, F^*)=r \quad , \quad \dim\alpha_\delta^{-1}(E)=\dim \Ext^1(\mathcal{O}_p, E)=r.$$
	Therefore 
\begin{displaymath} \dim \overline{B^1_{r, \delta+1, p}}=\dim B^1_{r, \delta}=\dim B^1_{r, d}=\dim B^1_{r, d+1}-1=\dim B^1_{r, \delta+1}-1,\end{displaymath} 
	by \cite[Theorem II.3.1]{S}, as required. 
	
As for irreducibility of $B^1_{r, \delta+1, p}$, it suffices to prove the irreduciblity of $B^1_{r, d, p}$. If $U\subset B^1_{r, d}$ is the set of $E\in B^1_{r, d}$ in which their general dual elementary transformations are stable, then $U$ is irreducible, as well. Since the fibers of $\alpha_\delta$ are irreducible, so $B^1_{r, \delta+1, p}=\nu_\delta(\alpha_\delta^{-1}(U))$ would be irreducible, as desired.
\end{proof}
\begin{remark}
\noindent
The content of Lemma \ref{Stability lemma} may already exist in the literature, but we were
unable to locate a suitable reference. A similar situation is studied in \cite[Lemma 1, Lemma 2]{Bis}. 
\end{remark}
\subsection{Excess dimension inequality} In this subsection, we apply the results of subsection  \ref{AS-Section} to Brill-Noether loci of stable rank two vector bundles.

For a general $p\in C$, consider $p$ as an effective divisor of degree $1$ on
$C$ and set $\mathcal{L}:=\mathcal{O}(\pi_1^*p)$. Then $\mathbb{L}:=(\pi_2)_*(\mathcal{L})$ is a line bundle on the  
sub-locus $\mathcal{U}:=B^0_{2, \delta+1, p}$. 
As parts (ii) and (iii) of lemma \ref{Stability lemma} show, the dense sub-locus $\mathcal{U}$ has non-empty intersection with 
 any irreducible component $X\subset B^k_{2, \delta+1}$.

Let  
\begin{align}\label{exacts}
0 \rightarrow [(\nu_\delta)_*(\alpha_\delta)^*(\mathbb{E}_\delta)] \stackrel{\alpha}\longrightarrow  \mathbb{F} \rightarrow  \mathbb{L} \rightarrow  0,
\end{align} 
be a general extension of $\mathbb{L}$ by $(\nu_\delta)_*(\alpha_\delta)^*(\mathbb{E}_\delta)$ and
observe that $\mathbb{F}$ is not necessarily a Poincar$\acute{e}$ bundle neither on $U(2, \delta+1)$ nor on $\mathcal{U}$. But it can be used to define a dense sub-locus of B-N loci on $\mathcal{U}$, on which we can compare the dimensions of $B^k_{2, \delta}$ and $B^k_{2, \delta+1}$. 

Now we consider a diagram of vector bundles on $\mathcal{U}$ as
{\small \begin{align}\label{diag0-0}
	\xymatrix {
		0\ar[r] & [(\nu_\delta)_*(\alpha_\delta)^*(\mathbb{E}_\delta)]_{\mid \mathcal{U}} \ar[r]^{\alpha}\ar[dr]^{\tilde{\gamma}_{\delta+1}} & (\mathbb{F})_{\mid \mathcal{U}}\ar[d]^{(\gamma_{\delta+1})_{\mid \mathcal{U}}}\ar[r]&\mathbb{L}\ar[r] & 0,\\
		&&[\mathbb{F}\otimes(\pi_2)_*\pi_1^*(\mathcal{O}_D(D))]_{\mid \mathcal{U}},&&
	}
	\end{align}
}
where $\tilde{\gamma}_{\delta+1}:=((\gamma_{\delta+1})_{\mid \mathcal{U}})\circ \alpha$. 
\begin{lemma}
For $k$ in the allowable range, we have 
\begin{align}
\dim D_{\delta-2(g-1)-k}(\tilde{\gamma}_{\delta+1})=\dim B^{k}_{2, d} \quad , \quad\dim D_{\delta-2(g-1)-k}((\gamma_{\delta+1})_{\mid \mathcal{U}})=\dim B^{k}_{2, d+1}.
\end{align}
\end{lemma}
\begin{proof}
At a point $e: 0 \rightarrow E \longrightarrow F \rightarrow  \mathcal{O}_p \rightarrow  0 \in \mathcal{U}$ the stalks of $[(\nu_\delta)_*(\alpha_\delta)^*(\mathbb{E}_\delta)]_{\mid \mathcal{U}}$, $(\mathbb{F})_{\mid \mathcal{U}}$ and $[\mathbb{F}\otimes(\pi_2)_*\pi_1^*(\mathcal{O}_D(D))]_{\mid \mathcal{U}}$ are $H^0(E)$, $H^0(F)$ and $H^0(F_D)$, respectively. Also $\ker (\tilde{\gamma}_{\delta+1})_e=H^0(E(-D))$. So 
 $\dim \theta^{-1}(D_{\delta-2(g-1)-k}(\tilde{\gamma}_{\delta+1}))=\dim B^{k}_{2,d}$. This gives the first equality and the proof for another equality is the same.  
\end{proof}
Now we apply Theorem \ref{SDual} to this situation, by which we obtain:
\begin{theorem} \label{Main B-N1} If $B^{k}_{2, d}$ is nonempty, then
	\begin{enumerate}
		\item[(a)]\label{1} $\dim B^k_{2, d} \ \ge  \dim B^{k}_{2, d+1}-k$, 
		\item[(b)] $\dim B^{k}_{2, d}\geq \dim B^{k-1}_{2,d}-(2(g-1)-d+k),$
		\item[(c)] $\dim B^{k}_{2, d+1}\geq \dim B^{k-1}_{2,d}-(2(g-1)-d+2k).$ 
	\end{enumerate}
	In particular, if $B^{k}_{2, d}$ is non-empty, then the scheme $B^{k}_{2, d+1}$ is of expected dimension if and only if $B^{k}_{2, d}$ is of expected dimension.
\end{theorem}
\begin{proof}
	(a) Taking the inequality (\ref{1}) of theorem (\ref{SDual}) into account, 
	it suffices to prove that no irreducible component of $B^{k}_{2, d}$ is contained in $B^{k+1}_{2, d}$, which is immediate by \cite[Remark 2.3]{CM-T} or by \cite[Prop. 1.6]{Laumon}. 
	
Part (b) is a direct consequence of part (a) together with the Riemann-Roch theorem and part (c) is a combination of (a) and (b).
\end{proof}
 \textbf{Example 1:} 
 There are cases in which the inequalities in theorem \ref{Main B-N1} get to be sharp. 
On general curves of genus $g\geq 1$ and for $d$ with $3\leq d \leq 2(g-1)$, M. Teixidor i Bigas computes the dimension of $B^2_{2, d}$ to be $2d-3$ and $B^3_{2, d}$ to be either empty or of dimension $3d-2g-6$. See 
\cite{T90} and \cite[Theorem 2]{T92}. So, the equality holds for such $B^2_{2, d}$, $B^2_{2, d+1}$, $B^3_{2, d}$ and $B^3_{2, d+1}$, in the allowable range. 

Also inequality can hold in theorem \ref{Main B-N1} strictly. Let $C$ be a general curve of genus $6$. Then, the scheme $B^4_{2, 10}$ is non-empty, because the loci 
\[B^4_{2, K}:=\{ E\in B^4_{2, 10}\mid \det (E)=K\},\]
is non-empty. Furthermore, $B^4_{2, 10}$
 would have superabundant components of dimension at least $6$, provided that $B^4_{2, K}$
  is strictly contained in $B^4_{2, 10}$. Assuming this to hold, we would have $\dim B^4_{2, 10}\geq 6$. Once again, using \cite{T90} and \cite[Theorem 2]{T92}, we obtain $\dim B^3_{2, 9}=9$. So \ref{Main B-N1}(c) holds strictly. Furthermore $\dim B^4_{2, 11}=\dim B^3_{2, 9}=9$, by residuation. So \ref{Main B-N1}(b) holds strictly.
  
  It remains only to prove that not any vector bundle $E\in B^4_{2, 10}$ belongs to $B^4_{2, K}$, to which it suffices to represent a bundle $E\in B^4_{2, 10}$ with $\det (E)\neq K$. In order to see this, take line bundles $L_1, L_2\in W^1_{5}\setminus W^2_{5}$ with $L_1+L_2\neq K$ and consider the extensions as 
  \[0\rightarrow L_1 \rightarrow E\rightarrow L_2\rightarrow 0.\] 
 Such an extension, belonging to $B^4_{2, 10}\setminus B^4_{2, K}$, exists according to \cite[Proposition 3.1]{L-Na} and \cite[Lemma 2.18]{L-Ne}.

 \textbf{Example 2:} 
This example shows that Lemma \ref{Stability lemma}(ii) fails if the genericity of the dual elementary transformation was removed. Suppose $E$ is a rank two bundle with $\deg(E)=1$ and has a line sub bundle of degree $0$. Take a generic point $p\in C$ and let $F$ be the dual elementary transformation of $E$ by $p$ satisfying $\Ker(E_p\rightarrow F_p)=L_{\mid p}$. Then for $p_1, \cdots, p_t \in C$ with $p_i\neq p$ we have a diagram as 
	{\small
	\[ \xymatrix{  0\ar [r]& E(p_1+\cdots + p_t) \ar[r]&  F(p_1+\cdots + p_t) \ar[r] & \mathcal{O}_p \ar[r] & 0 \\
		0 \ar[r]&L(p_1+\cdots + p_t) \ar[u]\ar[r] & L(p_1+\cdots + p_t+p) \ar[u] \ar[r] & \mathcal{O}_p\ar[u]^{=} \ar[r] & 0.} \] }
 Then, $\deg L(p_1+\cdots + p_t+p)\geq \mu(F(p_1+\cdots + p_t))$, so $F$ is strictly semi-stable.

\subsection{A byproduct:} As a consequence, we relate the smoothness    
 of rank two Brill-Noether schemes with different degrees.  
 In order to do this, we first make a comparison among the dimension of the tangent spaces of these schemes. See also  \cite[Th. 3.1]{Baj18}, \cite{Cop.} and \cite{HHN18}.
\begin{lemma}\label{lemma 3-3}
	Let $p\in C$ be a general point, $E\in B^{k}_{r, d}\setminus B^{k+1}_{r, d}$ and $F$ be a general stable dual elementary transformation of $E$ by $p$.
	Then, for $1\leq d\leq r(g-1)$ and $1\leq k\leq \frac{d}{2}+1$ we have
\begin{displaymath}
\dim T_E(B^{k}_{r, d})\geq \dim T_F(B^{k}_{r, d+1})-k.
\end{displaymath}
\end{lemma}
\begin{proof}
	Recall that
\begin{displaymath}
	{\small 
		T_E(B^{k}_{r, d})=
		[\im \{\mu_E: H^0(E)\otimes H^0(K\otimes E^*)\rightarrow H^0(K\otimes E\otimes E^*)\}]^\perp,
	}
	\end{displaymath}
	where for a sub-vector space $W
	\leq V$, by $W^\perp$ we mean the annihilator vector space of $W$ inside the dual of $V$ and $\mu_E$ denotes the Petri map of $E$.
	
	Notice that, under the hypothesis, the bundle $E$ is special. So lemma \cite[2.5]{N} can be applied to obtain $H^0(E)=H^0(F)$, therefore $F\in B^{k}_{r, d+1}\setminus B^{k+1}_{r, d+1}$. There exists now a diagram as
	{\small
		\[ \xymatrix{ H^0(E)\otimes H^0(K\otimes E^*) \ar[r]^{\mu_E} & H^0(K\otimes E\otimes E^*)\ar[dr]^{\alpha} &  \\
			H^0(F)\otimes H^0(K\otimes F^*) \ar[r]^{\mu_F}\ar[u]^{i} &H^0(K\otimes F\otimes F^*)\ar[r]^{\beta}& H^0(K\otimes F\otimes E^*),} \] }
	where the morphisms $i$, $\alpha$ and $\beta$ are all injective. Therefore, we will have $\dim \ker(\mu_E) \geq \dim \ker(\mu_F)$.
The Riemann-Roch theorem implies that $h^0(K\otimes F^*)=h^0(K\otimes E^*)-1$. Therefore,
	\begin{displaymath} 
	kh^0(K\otimes E^*)-\dim \ker (\mu_E)\leq kh^0(K\otimes F^*)-\dim \ker (\mu_F)+k,\end{displaymath} 
	implying $\dim \im (\mu_E)\leq \dim \im (\mu_F)+k$, as required. 
\end{proof}
\begin{corollary}\label{cor}
Let $E\in B^{k}_{2, d}$ and $F$ be a general dual elementary transformation of $E$ by a general point $p\in C$. Assume moreover that $B^{k}_{2, d}$ is of expected dimension. Then, for $1\leq d\leq 2(g-1)$ and $1\leq k\leq \frac{d}{2}+1$ the scheme $B^{k}_{2, d}$ is smooth at $E$ if and only if $B^{k}_{2, d+1}$ is smooth at $F$.
\end{corollary}
\begin{proof}
 Since $B^{k}_{2, d}$ is of expected dimension so is the scheme $B^{k}_{2, d+1}$, and vice versa; by theorem \ref{Main B-N1}. The smoothness of $B^{k}_{2, d}$ at $E\in B^{k}_{2, d}$ implies that $\dim B^{k}_{2, d}=\dim T_EB^{k}_{2, d}$. This, by Lemma \ref{lemma 3-3} together with the fact that $B^{k}_{2, d+1}$ is of expected dimension, gives rise to the assertion. 
\end{proof}
\textbf{Acknowledgement:} George H. Hitching read a previous version of this paper carefully and made various useful comments, I express my gratitude to him. I am also deeply grateful to Edoardo Ballico and Peter Newstead for their useful hints and responses to my questions.  
	This research was in part supported by a grant from IPM (No. 1400140038).

\end{document}